\documentclass{amsart}
\usepackage[english]{babel}
\usepackage{amsfonts, amsmath, amsthm, amssymb,amscd,indentfirst}

\usepackage{esint}
\usepackage{amsmath,amssymb,latexsym,indentfirst}
\usepackage{times}
\usepackage{palatino}

\newtheorem{theorem}{Theorem}[section]

\newtheorem{proposition}{Proposition}[section]

\newtheorem{definition}{Definition}[section]

\newtheorem{corollary}{Corollary}[section]

\newtheorem{remark}{Remark}[section]

\def\Jc{{\mathcal J}}

\def\Lc{{\mathcal L}}

\def\Kc{{\mathcal K}}

\def\Ac{{\mathcal A}}

\def\Ic{\mathcal{I}}

\def\Hc{{\mathcal H}}

\begin{document}

\title[A Penrose inequality]{A Penrose inequality for  asymptotically locally hyperbolic graphs}

\author{Levi Lopes de Lima}
\address{Federal University of Cear\'a,
Department of Mathematics, Campus do Pici, Av. Humberto Monte, s/n, Bloco 914, 60455-760,
Fortaleza/CE, Brazil.}
\email{levi@mat.ufc.br}
\email{fred@mat.ufc.br}
\author{Frederico Gir\~ao}
\thanks{The first author was partially supported by Fapesp/SP Grant 2012/09925-4}


\begin{abstract}
We use the inverse mean curvature flow to prove a sharp Alexandrov-Fenchel-type inequality for a class of hypersurfaces in certain locally hyperbolic manifolds. As an application we derive an optimal Penrose inequality for asymptotically locally hyperbolic graphs in any dimension $n\geq 3$. When the horizon has the topology of a compact surface of genus at least one, this provides an affirmative answer, for this class of initial data sets, to a question posed by Gibbons, Chru\'sciel and Simon on the validity of a Penrose-type inequality for black hole solutions carrying a higher genus horizon.
\end{abstract}

\maketitle

\section{Introduction and statements of the results}
\label{intro}

Penrose-type inequalities  express the  widely accepted view in General Relativity that a black hole should contribute to the total mass of an isolated gravitational system through its size as measured by the area of its event horizon \cite{BC} \cite{Ma}. In the asymptotically flat case, which corresponds to the cosmological constant $\Lambda=0$, this inequality has been established for time-symmetric initial data sets by Huisken-Ilmanen \cite{HI} and Bray \cite{Br} in the physical dimension $n=3$ and by Bray-Lee \cite{BL} for $n\leq 7$. Recent contributions to the graph and conformally flat cases in all dimensions appear in \cite{L} \cite{dLG1} \cite{dLG2} \cite{FS} \cite{HW} \cite{MV} for the ADM mass and in \cite{GWW1} \cite{GWW2} \cite{LWX} for the Chern-Gauss-Bonnet mass.

Recently, there has been much interest in extending this type of result to asymptotically hyperbolic (AH) manifolds, corresponding to the case $\Lambda <0$.
In particular, the conjectured Penrose inequality has been established when the manifold can be appropriately embedded as a graph in hyperbolic space \cite{dLG3} \cite{dLG4}; see also \cite{DGS} for a previous contribution. The purpose of this note is to extend the method in \cite{dLG4} to asymptotically locally hyperbolic (ALH) graphs, thus proving an optimal Penrose inequality for such manifolds in any dimension $n\geq 3$; see Theorem \ref{theomain}.
In the physical dimension $n=3$, when the horizon has the  topology of a compact surface of genus $\gamma\geq 1$, our result provides, for this class of initial data sets, an affirmative answer (Corollary \ref{cormain}) to  a question posed by Chru\'sciel-Simon \cite{CS}  and Gibbons \cite{G} on the validity of a Penrose-type inequality for black hole solutions carrying a higher genus horizon; see also \cite{BC} \cite{Ma} for background on this problem. A key ingredient in the proof of Theorem \ref{theomain} is a sharp Alexandrov-Fenchel-type inequality for a class of hypersurfaces in certain locally hyperbolic manifolds; see Theorem \ref{aftype}. This result has an independent interest and its proof uses the inverse mean curvature flow and an integral inequality due to Brendle \cite{Br}.
In this regard we mention that lately there has been considerable progress in establishing Alexandrov-Fenchel-type inequalities in hyperbolic space; see for instance \cite{WX} and the references therein.

We now proceed to a detailed account of our main results. A mass-type invariant for ALH manifolds has been introduced by Chru\'sciel, Herzlich and Nagy \cite{CH} \cite{CN} \cite{H}. Their definition applies in particular to the class of ALH manifolds we consider here. We start by describing the corresponding {\em locally hyperbolic} (LH) reference metrics. Fix $\epsilon =-1, 0$ and let $(N^{n-1},h)$ be a closed space form of sectional curvature $\epsilon$. In the product manifold $P_\epsilon=I_\epsilon\times N$, where $I_{-1}=(1,+\infty)$ and $I_0=(0,+\infty)$, define the metric
\begin{equation}\label{half}
g_\epsilon=\frac{dr^2}{\rho_{\epsilon}(r)^2}+r^2h, \quad r\in I_\epsilon,
\end{equation}
where
\[
\rho_\epsilon (r)=\sqrt{r^2+\epsilon}.
\]
It is easy to check that $(P_\epsilon,g_\epsilon)$ is {locally hyperbolic} in the sense that its sectional curvature is constant and equal to $-1$.
For further reference we also note that the manifold
$Q_\epsilon=\mathbb R\times P_\epsilon$, endowed with the metric
\begin{equation}\label{aswe}
\overline g_\epsilon=\rho_\epsilon^2d\tau^2+g_\epsilon,
\end{equation}
is locally hyperbolic as well.
Moreover, if $\tau_0\in\mathbb R$ then the horizontal slice $P_{\epsilon}^{\tau_0}\subset Q_\epsilon$ given by $\tau=\tau_0$ is totally geodesic. This follows from the fact that the vertical vector field $\partial_t$ is Killing. In particular, each $P_\epsilon^{\tau_0}$ can be naturally identified to $P_\epsilon$.

Let $(M^n,g)$ be a complete $n$-dimensional Riemannian manifold, possibly carrying a compact inner boundary $\Sigma$. For simplicity, we assume that $M$ has a unique end, say $E$. Following \cite{CH} we say that $(M,g)$ is {\em asymptotically locally hyperbolic} (ALH) if there exists a chart $\Psi$ taking $E$ to the end of $P_\epsilon$ corresponding to $r=+\infty$
so that, as $r\to +\infty$,
\begin{equation}\label{alhcond}
\|\Psi_*g-g_{\epsilon}\|_{g_\epsilon}+ \|d\Psi_*g\|_{g_\epsilon}=O\left(r^{-\sigma}\right),
\end{equation}
for some $\sigma>n/2$. We also assume that $R_{\Psi_*g}+n(n-1)$ is integrable, where $R$ denotes scalar curvature. For this kind of manifold,  a mass-like invariant can be defined as
\begin{equation}\label{massl}
\mathfrak m_{(M,g)}=\lim_{r\to +\infty}c_n\int_{N_r}
  \left(\rho_\epsilon({\rm div}_{g_{\epsilon}}e-d{\rm tr}_{g_{\epsilon}}e)-i_{\overline \nabla\rho_\epsilon}e+
  ({\rm tr}_{g_{\epsilon}}e)d\rho_\epsilon\right)(\nu_r)dN_r,
\end{equation}
where $e=\Psi_*g-g_{\epsilon}$, $N_r=\{r\}\times N$, $\overline\nabla$ is the connection of $g_\epsilon$, $\nu_r$ is the {\em outward} unit vector to $N_r$, which means that $\nu_r$ points toward  the  end $r=+\infty$, and
\[
c_n=\frac{1}{2(n-1)\vartheta_{n-1}},
\]
where $\vartheta_{n-1}={\rm area}_{n-1}(N)$. We emphasize that a crucial point in justifying (\ref{massl}) is precisely to check that the right-hand side does not depend on the chart at infinity $\Psi$ satisfying (\ref{alhcond}). In any case, $\mathfrak m_{(M,g)}$ is an invariant of the asymptotic geometry of $(M,g)$ which measures the average rate of the convergence of $g$ to $g_\epsilon$ as $r\to +\infty$; details can be found in \cite{CH}.

\begin{remark}\label{epsilon1}
{\rm The above construction also applies to the case in which $\epsilon=1$ and $(N,h)$ is a round sphere, so that a mass invariant is also available. This is the asymptotically hyperbolic case, treated in \cite{dLG4}.}
\end{remark}

An ALH manifold $(M,g)$ as above can be thought of as the initial data set of a time-symmetric solution of the Einstein field equations with negative cosmological constant. The invariant $\mathfrak m_{(M,g)}$ is then interpreted as the total mass of the solution. Physical reasoning predicts  that $\mathfrak m_{(M,g)}$ should have the appropriate sign under the relevant dominant energy condition, namely, $R_g\geq -n(n-1)$. When $M$ carries a compact minimal horizon $\Sigma$ one expects the invariant to satisfy a Penrose-type inequality in the sense that  it should be bounded from below by a suitable expression involving the area $|\Sigma|$ of $\Sigma$. In order to figure out the correct form of this inequality, we consider the so-called {\em Kottler black hole metrics}, which are rather straightforward deformations of the LH metrics $g_\epsilon$ above.

Let us introduce a real parameter $m>0$
and consider the metric
\begin{equation}\label{bhmet}
g_{\epsilon,m}=\frac{dr^2}{\rho_{\epsilon,m}(r)^2}+r^2h,
\end{equation}
where
\[
\rho_{\epsilon,m}(r)=\sqrt{r^2+\epsilon-\frac{2m}{r^{n-2}}}.
\]
For each $m$ as above, it is easy to see that  the function
\[
r\mapsto f_{\epsilon,m}(r)=r^n+\epsilon r^{n-2}-2m
\]
is strictly positive for  $r>r_{\epsilon,m}$, where $r_{\epsilon, m}$ is the unique positive zero of $f_{\epsilon,m}$. A further analysis shows that the function $m\mapsto r_{\epsilon,m}$ is smooth and strictly increasing with range $(0,+\infty)$ if $\epsilon=0$ and $(1,+\infty)$ if $\epsilon=-1$.

This discussion shows that the metric $g_{\epsilon,m}$ is well defined on the product $P_{\epsilon,m}=I_{\epsilon,m}\times N$, where  $I_{\epsilon,m}=\{r;r>r_{\epsilon,m}\}$. Moreover, it extends smoothly to the slice $r=r_{\epsilon,m}$, the so-called {\em horizon}, denoted $\Hc_{\epsilon,m}$. This terminology can be justified as follows. The metric $g_{\epsilon,m}$ is {\em static} in the sense that
\begin{equation}\label{static}
(\overline\Delta \rho_{\epsilon,m})g_{\epsilon,m}-\overline\nabla^2
\rho_{\epsilon,m}+\rho_{\epsilon.m}{\rm Ric}_{g_{\epsilon.m}}=0,
\end{equation}
where $\overline \nabla$ is the connection of $g_{\epsilon,m}$ and $\overline \Delta$ is the corresponding Laplacian.
This means that the Lorentzian metric
\[
\tilde g_{\epsilon,m}=-\rho_{m,\epsilon}^2d\tau^2+g_{\epsilon,m},
\]
defined on $Q_{\epsilon,m}=\mathbb R\times P_{\epsilon,m}$,
is a static solution to the vacuum  field equations with negative cosmological constant:
\[
{\rm Ric}_{\tilde g_{\epsilon,m}}=-n\tilde g_{\epsilon,m}.
\]
Moreover, the null hypersurface $r=r_{\epsilon,m}$
defines the event horizon surrounding the central singularity $r=0$. This justifies the horizon terminology and explains why  $g_{\epsilon,m}$ is termed a black hole metric.
We note in passing that $\overline g_\epsilon$ in (\ref{aswe}) is obtained from $\tilde g_{\epsilon,0}$ by Wick rotation: $\tau\mapsto i\tau$.

A computation shows  that
if $(\theta_1,\cdots,\theta_{n-1})$ are coordinates in $N_r$ then the sectional curvatures of $g_{\epsilon,m}$ are
\[
K(\partial_r,\partial_{{\theta}_i})=-1-(n-2)\frac{m}{r^n}
\]
and
\[
K(\partial_{{\theta}_i},\partial_{{\theta}_j})=-1+\frac{2m}{r^{n-2}}.
\]
This not only shows that $g_{\epsilon,m}$ satisfies the appropriate dominant energy condition, since
its scalar curvature is
\[
R_{g_{\epsilon,m}} =  -n(n-1),
\]
but also suggests that $g_{\epsilon,m}$ is ALH in the sense described above. In fact, a straightforward computation  gives
\[
\|g_{\epsilon,m}-g_{\epsilon}\|_{g_\epsilon}+ \|dg_{\epsilon,m}\|_{g_\epsilon}=O\left(mr^{-n}\right),
\]
as expected.
Using (\ref{massl}) we finally conclude that $\mathfrak m_{(P_{\epsilon,m},g_{\epsilon,m})}=m$, which shows that $m$ should be interpreted as the total mass of $g_{\epsilon,m}$.

One immediately finds that
the area  $|\Hc_{\epsilon,m}|$ of the horizon $\Hc_{\epsilon,m}$ of $(P_{\epsilon,m},g_{\epsilon,m})$ relates to its mass  $m$ by means of
\[
m=\frac{1}{2}\left(\left(\frac{|\Hc_{\epsilon,m}|}{\vartheta_{n-1}}\right)^{\frac{n}{n-1}}
+\epsilon \left(\frac{|\Hc_{\epsilon,m}|}{\vartheta_{n-1}}\right)^{\frac{n-2}{n-1}}\right).
\]
Thus, in analogy with the standard Penrose inequality, it is natural to conjecture  that if $(M,g)$ is an $n$-dimensional ALH manifold (with respect to the reference metric $g_\epsilon$) carrying an outermost minimal  horizon $\Sigma$ and satisfying $R_g\geq -n(n-1)$ everywhere then there holds
\begin{equation}\label{propcon}
\mathfrak m_{(M,g)}\geq\frac{1}{2}\left(\left(\frac{|\Sigma|}{\vartheta_{n-1}}\right)^{\frac{n}{n-1}}
+\epsilon \left(\frac{|\Sigma|}{\vartheta_{n-1}}\right)^{\frac{n-2}{n-1}}\right),
\end{equation}
with the equality occurring if and only if $(M,g)$ is isometric to the corresponding black hole metric.

In the physical dimension $n=3$, (\ref{propcon}) first appears in \cite{CS} as a conjectured inequality whose veracity would follow in case the use of the so-called Geroch's monotonicity of the Hawking mass under the inverse mean curvature flow, as envisaged by Gibbons \cite{G}, could be justified.
Contrary to this initial expectation, Neves \cite{N} has shown that, at least in the AH case, the convergence properties of the flow at infinity are insufficient to implement Geroch's scheme. Similar remarks should also apply in the general ALH context, even though Lee and Neves \cite{LN1}  \cite{LN2} have recently established that Geroch's strategy works in the so-called \lq non-positive mass range\rq; see Remark \ref{negmass}.
Despite these negative results, we show here that the method developed in \cite{dLG3} \cite{dLG4} can be adapted to handle the special case of graphs in {\em any} dimension $n\geq 3$. To explain this we observe that each $(P_{\epsilon,m},g_{\epsilon,m})$ can be isometrically immersed as a radially symmetric vertical graph inside $(Q_{\epsilon},\overline g_\epsilon)$: the defining function $u_{\epsilon,m}:I_{\epsilon,m}\to \mathbb R$ satisfies $u_{\epsilon,m}(r_{\epsilon,m})=0$ and
\begin{equation}\label{embed}
\rho_{\epsilon}(r)^2\left(\frac{du_{\epsilon,m}}{dr^2}\right)^2
=\frac{1}{\rho_{\epsilon,m}(r)^2}-\frac{1}{\rho_{\epsilon}(r)^2}.
\end{equation}
It is clear from this that the horizon $\Hc_{\epsilon,m}$ lies on the totally geodesic horizontal slice $P_\epsilon^{0}$, with the intersection $M\cap P_\epsilon^{0}$ being orthogonal along $\Hc_{\epsilon,m}$. This motivates us to consider a more general class of hypersurfaces in $(Q_{\epsilon},\overline g_\epsilon)$.

\begin{definition}\label{alshyp}
We say that a complete hypersurface $(M,g)\subset (Q_{\epsilon},\overline g_\epsilon)$, possibly carrying a compact inner boundary $\Sigma$, is {\em asymptotically locally hyperbolic} if there exists a compact set $K\subset M$ so that $M-K$ can be written as a graph over the end $E_0$ of the horizontal slice
 $P_{\epsilon}^{0}\subset Q_{\epsilon}$, with the graph being  associated to a smooth function $u$ such
that the asymptotic condition (\ref{alhcond}) holds for the nonparametric chart $\Psi_u(x,u(x))=x$, $x\in E_0$.
Moreover, we assume that $R_{{\Psi_u}_*g}+n(n-1)$ is integrable.
\end{definition}

Under these conditions, the mass $\mathfrak m_{(M,g)}$ can be computed by taking $\Psi=\Psi_u$ in (\ref{massl}). More precisely, if we assume that the inner boundary $\Sigma$ lies on some totally geodesic, horizontal slice $P_{\epsilon}^{\tau_0}$, which we of course identify with $P_\epsilon$, and moreover that the intersection  $M\cap P^{\tau_0}_\epsilon$ is orthogonal along $\Sigma$, so that $\Sigma\subset M$ is minimal and hence a horizon indeed, then the computations in \cite{dLG2} actually give the following integral formula for the mass:
\begin{equation}\label{massform}
\mathfrak m_{(M,g)}=c_n\int_M\Theta\left(R_g+n(n-1)\right) dM+
c_n\int_\Sigma \rho_{\epsilon} Hd\Sigma,
\end{equation}
where $\Theta=\langle \partial/\partial t,N\rangle$, $N$ is the unit normal to $M$, which we choose so as to point upward at infinity, and $H$ is the mean curvature of $\Sigma\subset P_{\epsilon}$ with respect to its {\em inward} pointing unit normal, which means that the normal points in the direction {\em  opposite} to the end of $P_{\epsilon}$ given by $r=+\infty$. In particular, if $R_g\geq -n(n-1)$ and $M$ is a graph ($\Theta>0$) then
\begin{equation}\label{massineq}
\mathfrak m_{(M,g)}\geq c_n\int_\Sigma \rho_{\epsilon} Hd\Sigma.
\end{equation}

We are now in a position to use the following Alexandrov-Fenchel-type inequality.

\begin{theorem}\label{aftype}
Let $\Sigma\subset P_\epsilon$ be a compact embedded hypersurface which is {\em star-shaped} in the sense that it can be written as a radial graph over a slice $N_r$ and {\em strictly mean convex} in the sense that its mean curvature satisfies $H>0$. Then there holds
\begin{equation}\label{aftype2}
c_n\int_\Sigma \rho_{\epsilon} Hd\Sigma\geq \frac{1}{2}\left(
\left(\frac{|\Sigma|}{\vartheta_{n-1}}\right)^{\frac{n}{n-1}}
+\epsilon \left(\frac{|\Sigma|}{\vartheta_{n-1}}\right)^{\frac{n-2}{n-1}}\right),
\end{equation}
with the equality occurring if and only if $\Sigma$ is a slice.
\end{theorem}

By combining (\ref{massineq}) and (\ref{aftype2}) we immediately obtain the first part of our main result.

\begin{theorem}\label{theomain}
If $M\subset Q_{0,\epsilon}$ is an ALH graph as above, so that its horizon $\Sigma\subset P^{\tau_0}_\epsilon$ is
star-shaped and  mean convex in the sense that $H\geq 0$, then
\begin{equation}\label{penrose}
\mathfrak m_{(M,g)}\geq \frac{1}{2}\left(
\left(\frac{|\Sigma|}{\vartheta_{n-1}}\right)^{\frac{n}{n-1}}
+\epsilon\left(\frac{|\Sigma|}{\vartheta_{n-1}}\right)^{\frac{n-2}{n-1}}\right),
\end{equation}
with the equality holding if and only if $(M,g)$ is (congruent to) the graph realization of the corresponding black hole solution.
\end{theorem}

\begin{remark}\label{effec}
{\rm Under the conditions of Theorem \ref{theomain}, the mass $\mathfrak m_{(M,g)}$ is always positive due to (\ref{massineq}). In particular, if $\epsilon=-1$ the lower bounds (\ref{aftype2}) and (\ref{penrose}) only are  effective if we further assume that  $|\Sigma|>\vartheta_{n-1}$.}
\end{remark}

As already observed, the following corollary provides a positive answer to   a question posed by Gibbons \cite{G} and Chru\'sciel-Simon \cite{CS} for the class of initial data sets we consider.

\begin{corollary}\label{cormain}
If the horizon $\Sigma$ is a surface of genus $\gamma\geq 1$
then
\begin{equation}\label{haw}
\mathfrak m_{(M,g)}\geq \left(\frac{4\pi}{\vartheta_2}\right)^{{3}/{2}}
\sqrt{\frac{|\Sigma|}{16\pi}}\left(1-\gamma+\frac{|\Sigma|}{4\pi}\right),
\end{equation}
where for $\gamma=1$ we assume the normalization $\vartheta_2=4\pi$.
Moreover, the equality holds if and only if $(M,g)$ is (congruent to) the graph realization of the corresponding black hole solution.
\end{corollary}

\begin{proof}
If $\gamma\geq 2$ this follows by taking $n=3$ and $\epsilon=-1$ in the theorem and  observing that Gauss-Bonnet gives $\vartheta_2=4\pi (\gamma-1)$. If $\gamma=1$ we take $\epsilon=0$ and use the normalization.
\end{proof}

\begin{remark}\label{negmass}
{\em As discussed in \cite{CS}, when $\epsilon=-1$ the Kottler metrics (\ref{bhmet}) also describe static black hole solutions when the parameter $m$ becomes negative up to a certain critical value, namely,
\[
m_{\rm crit}=-\frac{(n-2)^{\frac{n-2}{2}}}{n^{\frac{n}{2}}}.
\]
In this regard we mention that Lee and Neves \cite{LN1} \cite{LN2} used the Huisken-Ilmanen's formulation of the inverse mean curvature flow to establish a Penrose-type inequality for conformally compact ALH $3$-manifolds in this mass range. More precisely, they prove (\ref{haw}) with the mass replaced by the supremum of the mass aspect function, which is assumed to be non-positive along the boundary at infinity. In particular, their manifolds always have {\em non-positive} mass while our graphs necessarily satisfy $\mathfrak m_{(M,g)}>0$; see Remark \ref{effec}. Thus, their result and Corollary \ref{cormain} are in a sense complementary to each other.}
\end{remark}

The proofs of Theorems \ref{aftype} and \ref{theomain} are presented in Section \ref{proof}. This uses the asymptotics of solutions of the mean curvature flow established in Section \ref{imcflh}.

\vspace{0.2cm}

\noindent
{\bf Acknowledgements:} The authors would like to thank Prof. A. Neves for making available the preprint \cite{LN1}.

\section{The inverse mean curvature flow in LH manifolds}
\label{imcflh}

In this section we study the asymptotics of the inverse mean curvature for a class of hypersurfaces in
$P_\epsilon$. Our main result, Theorem \ref{longtime}, follows from standard arguments in the theory of such flows; see for instance \cite{Ge1}, \cite{Ge2} , \cite{BHW} and \cite{D}. For convenience we include all the relevant estimates here, even though our treatment is necessarily brief. The interested reader will find a more detailed account in the references above.

\subsection{Hypersurfaces in LH manifolds}\label{hyperlh}
It is convenient to consider a parameter $s$ so that
\[
ds=\frac{dr}{\rho_{\epsilon}(r)},
\]
which gives
\[
s=\left\{
\begin{array}{cc}
\log r, & \epsilon =0\\
\log (2\sqrt{r^2-1}+2r), & \epsilon=-1
\end{array}
\right.
\]
Thus, $s\in\mathbb R$ if $\epsilon=0$ and $s>\log 2$ if $\epsilon=-1$.
Using this coordinate, it follows from (\ref{half}) that
\begin{equation}\label{change}
g_{\epsilon}=ds^2+\lambda_\epsilon(s)^2h,
\end{equation}
where
\[
\lambda_\epsilon(s)=\left\{
\begin{array}{cc}
e^s, & \epsilon =0\\
\frac{e^{s}}{4}+e^{-s}, & \epsilon=-1
\end{array}
\right.
\]
We note the useful identities
\begin{equation}\label{usef}
\dot\lambda^2_\epsilon=\lambda^2_\epsilon+\epsilon, \quad \ddot\lambda_\epsilon=\lambda_\epsilon.
\end{equation}
Also, in terms of $s$ we have
\[
\rho_\epsilon(s)=\left\{
\begin{array}{cc}
e^s, & \epsilon =0\\
\sqrt{\left(\frac{e^s}{4}+e^{-s}\right)^2-1}, & \epsilon=-1
\end{array}
\right.
\]
This shows that
\begin{equation}\label{fund}
\dot\lambda_\epsilon=\rho_\epsilon.
\end{equation}

We consider compact hypersurfaces $\Sigma\subset P_{\epsilon}$ which are {\em star-shaped} in the sense that they can be written as  radial graphs over some slice $N_s=\{s\}\times N$. We choose the {\em inward} unit normal $\xi$ to $\Sigma$, which means that $\xi$ points toward the direction {\em opposite} to the end of $P_\epsilon$ defined by $s=+\infty$. Thus, if $\Sigma$ is graphically represented as
\[
\theta\in N\mapsto (u(\theta),\theta)\in P_\epsilon,
\]
then
\[
\xi=\frac{1}{W}\left(\frac{u^i}{\lambda_\epsilon(u)^2}\partial_{\theta_i}
-\partial_s\right),\quad W=\sqrt{1+|D v|_h^2},
\]
where $u_i=\partial_{\theta_i}u$, $u^i=h^{ij}u_j$, $D$ is the covariant derivative of $h$ and
\begin{equation}\label{change2}
v=\varphi(u), \quad \dot \varphi=1/\lambda_\epsilon.
\end{equation}
We note that the tangent space to the graph is spanned by
\[
Z_i=\lambda_\epsilon v_i\partial_s+\partial_{\theta_i},\quad i=1,\cdots,n-1,
\]
so that the induced metrics are
\begin{equation}\label{met1}
\eta_{ij}=\lambda_\epsilon^2(h_{ij}+v_iv_j)
\end{equation}
and
\[
\eta^{ij}=\frac{1}{\lambda_{\epsilon}^{2}}\left(h^{ij}-\frac{v^iv^j}{W^2}\right).
\]
The second fundamental form of the graph is
\[
c_{ij}=\frac{\lambda_\epsilon}{W}\left(\dot\lambda_\epsilon(h_{ij}+v_iv_j)-
v_{ij}\right),
\]
so  the shape operator is
\[
a^i_j=\eta^{ik}c_{kj}=\frac{\dot\lambda_\epsilon}{W\lambda_\epsilon}\delta^i_j-
\frac{1}{W\lambda_\epsilon}\tilde h^{ik}v_{kj},
\]
where
\[
\tilde h^{ij}=h^{ij}-\frac{v^iv^j}{W^2}.
\]
Thus, the mean curvature is
\begin{equation}\label{meancur}
H=a^i_i=(n-1)\frac{\dot\lambda_\epsilon}{W\lambda_\epsilon}-\frac{\tilde h^{ik}}{W\lambda_\epsilon}v_{ki}.
\end{equation}
We recall that
\begin{equation}\label{cauchy}
|a|^2\geq \frac{H^2}{n-1},
\end{equation}
with the equality holding if and only if $\Sigma$ is totally umbilical.

We also consider the {\em support function}
\begin{equation}\label{supp}
p=-\langle \overline \nabla\rho_\epsilon,\xi\rangle=\frac{\lambda_\epsilon}{W},
\end{equation}
which
relates to $\rho_\epsilon$ and $H$ by means of the following Minkowski-type identity:
\begin{equation}\label{minko}
\Delta \rho_\epsilon=(n-1)\rho_\epsilon -Hp,
\end{equation}
where $\Delta={\rm div}\circ \nabla$ is the Laplacian of $\Sigma$. This is an easy consequence of the fact that, for $m=0$, (\ref{static}) reduces to
\begin{equation}\label{confo}
\overline\nabla^2\rho_\epsilon=\rho_\epsilon g_\epsilon.
\end{equation}

\subsection{The inverse mean curvature flow}\label{subimcf}

We consider  a one-parameter family
\[
t\in [0,\delta)\mapsto \Sigma_t\subset P_\epsilon
\]
of star-shaped, strictly mean convex  hypersurfaces evolving under the inverse mean curvature flow.
This means that the corresponding parametrization $X:[0,\delta)\times N\to P_\epsilon$ satisfies
\begin{equation}\label{imcfpar}
\frac{\partial X}{\partial t}=-\frac{\xi}{H}.
\end{equation}
The following evolution equations are well-known \cite{D} \cite{dLG4}.

\begin{proposition}\label{evolequations}
Under the flow (\ref{imcfpar}) we have:
\begin{enumerate}
\item
The area element $d\Sigma$ evolves as
 \begin{equation}\label{evolvelemvol}
\frac{\partial}{\partial t}d\Sigma=d\Sigma.
\end{equation}
In particular, if $A$ is the area of $\Sigma$ then
\begin{equation}\label{evolvarea}
\frac{dA}{dt}=A;
\end{equation}
\item The shape operator evolves as
\begin{equation}\label{evolvshape}
\frac{\partial a}{\partial t}=\frac{\nabla^2H}{H^2}-2\frac{\nabla H\cdot\nabla H}{H^3}-\frac{a^2}{H}+\frac{I}{H},
\end{equation}
where $\cdot$ is the symmetric product and $I$ is the identity map. As a consequence,
the mean curvature evolves as
\begin{equation}\label{evolvmean}
\frac{\partial H}{\partial t}=\frac{\Delta H}{H^2}-2
\frac{|\nabla H|^2}{H^3} -\frac{|a|^2}{H}+\frac{n-1}{H},
\end{equation}
and the tensor $b=Ha$ evolves as
\begin{equation}\label{simmons}
\frac{\partial b}{\partial t} \approx \frac{\Delta b}{H^2}
 -2\frac{b^2}{H^2}+2(n-1)\frac{b}{H^2},
\end{equation}
where $\approx$ means that first order terms in $b$ have been omitted in the right-hand side.
 \item The static potential evolves as
 \begin{equation}\label{evolvstat}
 \frac{\partial \rho_\epsilon}{\partial t}=\frac{p}{H}.
 \end{equation}
 \item The support function evolves as
 \begin{equation}\label{evolvsupp}
 \frac{\partial p}{\partial t}=\frac{\rho_\epsilon}{H}+\frac{1}{H^2}
 \langle \nabla \rho_\epsilon,\nabla H\rangle.
 \end{equation}
 As a consequence of (\ref{minko}),
 \begin{equation}\label{evolvsuppint}
 \frac{d}{dt}\int_\Sigma p=n\int_\Sigma\frac{\rho_\epsilon}{H}d\Sigma.
 \end{equation}
\end{enumerate}
\end{proposition}

Up to tangential diffeomorphisms, (\ref{imcfpar}) means that
\begin{equation}\label{par}
\frac{\partial u}{\partial t}=\frac{W}{H},
\end{equation}
where $u=u(t,\cdot)$ is the evolving graphing function. Equivalently,
\begin{equation}\label{imcfnonpar}
\frac{\partial v}{\partial t}=\frac{W}{\lambda_\epsilon H},
\end{equation}
where $u$ and $v$ are related by (\ref{change2}).
It follows from (\ref{meancur}) that  (\ref{imcfnonpar})   is a  parabolic equation for $v$, so that a smooth solution always exists for any strictly mean convex initial hypersurface, at least for small $\delta$. We now develop the required apriori estimates to make sure that any such solution is defined for all time $t>0$.
Our final result, Theorem \ref{longtime}, is far from providing the optimal asymptotic behavior of solutions but it suffices for our purposes. Our presentation follows \cite{BHW} and \cite{Ge1} closely; see also \cite{D}.
As it is usually the case, a  key ingredient here is the Parabolic Maximum Principle.

In the following, whenever $U$ is a time-dependent function
we use the notation $U\sim e^{\frac{t}{n-1}}$ to mean that $U$ is bounded from above and below by a positive constant times $e^{\frac{t}{n-1}}$.
Also,
we set
\[
\underline U(t)=\inf_N U(t,\cdot), \quad \overline U(t)=\sup_N U(t,\cdot).
\]
The next $C^0$ estimate is an immediate consequence of the  fact that, in view of (\ref{meancur}), the mean curvature is comparable to
$
(n-1){\dot\lambda_\epsilon}/{W\lambda_\epsilon}
$
at an extremal point of $v$.

\begin{proposition}\label{c0est}
As long as the solution $u$ of (\ref{par}) exists, one has
\begin{equation}\label{c0estpre}
\lambda_\epsilon(u)\sim e^{\frac{t}{n-1}}.
\end{equation}
More precisely,
\begin{equation}\label{c0est2}
e^{\frac{t}{n-1}}\lambda_\epsilon(\underline u(0))\leq \lambda_{\epsilon}(\underline u(t))\leq \lambda_{\epsilon}(\overline u(t)) \leq e^{\frac{t}{n-1}}\lambda_\epsilon(\overline u(0)).
\end{equation}
\end{proposition}

We now turn to $C^1$ estimates. For this we first need  to control the quantity determining  the speed of the flow, namely, the mean curvature.

\begin{proposition}\label{conth}
As long as the mean curvature remains positive, it is bounded from above as follows:
\begin{equation}\label{conth2}
H\leq n-1 + O(e^{-\frac{2t}{n-1}}).
\end{equation}
\end{proposition}

\begin{proof}
It follows from (\ref{evolvmean}) and (\ref{cauchy}) that
\[
\frac{\partial H^2}{\partial t}\leq 2\frac{\Delta H}{H}-4
\frac{|\nabla H|^2}{H^2} -\frac{2}{n-1}{H^2}+2(n-1),
\]
which implies that $\varrho=H^2-(n-1)^2$ satisfies
\[
\frac{d \overline \varrho}{dt}\leq - \frac{2}{n-1}\overline \varrho.
\]
The result follows.
\end{proof}

If we differentiate  (\ref{imcfnonpar}) with respect to $v^kD_k$ then we obtain that the quantity
\begin{equation}\label{quantomeg}
\omega=\frac{|D v|_h^2}{2}
\end{equation}
evolves as
\begin{eqnarray}
\frac{\partial \omega}{\partial t} & = & \frac{\tilde h^{ij}}{W^2F^2}\omega_{ij}-
  \frac{1}{F^2}\frac{\partial F}{\partial v_i}\omega_i-2(n-2)\epsilon\frac{\omega}{W^2F^2}\nonumber\\
  & & \quad -\frac{1}{W^2F^2}\tilde h_{ij}h^{kl}v_{ij}v_{jl}-2(n-1) H^{-2}\omega,
\label{evolvomega}
\end{eqnarray}
where
\[
F=\frac{\lambda_\epsilon H}{W}.
\]

\begin{proposition}\label{apc1}
We have
\begin{equation}\label{apc12}
|D v|_h=O(e^{-\frac{t}{n-1}}),
\end{equation}
or equivalently,
\begin{equation}\label{apc123}
|D u|_h=O(1).
\end{equation}
In particular,
\begin{equation}\label{supernew}
W=1+O(e^{-\frac{2t}{n-1}}).
\end{equation}
\end{proposition}

\begin{proof}
From (\ref{evolvomega}), (\ref{conth2}) and (\ref{c0est2}),
\[
\frac{d\overline \omega}{dt}
   \leq  -\left(\frac{2}{n-1}+O(e^{-\frac{2t}{n-1}})\right)\overline \omega,
\]
and the assertion follows.
\end{proof}

The next proposition shows not only that strict mean convexity is preserved but also that the mean curvature remains bounded away from zero.

\begin{proposition}
The mean curvature is uniformly bounded from below by a positive constant.
\end{proposition}

\begin{proof}
First note that (\ref{c0est2}) and (\ref{apc12}) imply  $p \sim e^{\frac{t}{n-1}}$.
Also, we have that $\kappa=\log (1/pH)$ satisfies
\[
\frac{\partial \kappa}{\partial t}=\frac{\Delta \kappa}{H^2}-\frac{|\nabla \log p^{-1}|^2}{H^2}+
\frac{|\nabla \log H|^2}{H^2} -\frac{n-1}{H^2},
\]
so (\ref{conth2}) gives
\[
\frac{d\overline \kappa}{dt}\leq -\frac{1}{n-1}+O(e^{-\frac{2t}{n-1}}).
\]
Hence, $e^\kappa\leq O(e^{-\frac{t}{n-1}})$, which gives $H^{-1}\leq O(1)$.
\end{proof}

Now we are ready to establish $C^2$ apriori estimates for the flow.

\begin{proposition}\label{c2est}
The shape operator  remains uniformly bounded as long as the flow exists.
\end{proposition}

\begin{proof}
Let $\mu$ be the maximal eigenvalue of $b=Ha$. Clearly,
$|b|\leq C(\mu+1)$ for some constant $C>0$. On the other hand, (\ref{simmons}) gives
\[
\frac{d\overline \mu}{dt}\leq -2\frac{\overline \mu^2}{H^2}+2(n-1)\frac{\overline \mu}{H^2},
\]
so $\mu$ is uniformly bounded from above. Thus, $b$ is bounded and since $H$ is uniformly bounded, $a$ is bounded as well.
\end{proof}

The following result establishes the longtime existence for solutions of the inverse mean curvature flow (\ref{imcfpar}) and describes the asymptotics of the graphing function $u$ in a way that suffices for our purposes.

\begin{theorem}\label{longtime}
For any initial hypersurface which is star-shaped and strictly mean convex, the solution of  (\ref{imcfpar}) exists for all time $t>0$. Also, the evolving hypersurface expands toward infinity, remaining star-shaped and strictly mean convex. Moreover, there exists $\alpha\in \mathbb R$ so that the rescaled  function
\[
\tilde u=u-\frac{t}{n-1}
\]
converges to $\alpha$
in the sense that
\begin{equation}\label{sharpest}
|D \tilde u|_h+|D^2 \tilde u|_h=o(1).
\end{equation}
\end{theorem}

The longtime existence follows by applying Parabolic Regularity Theory \cite{Ge2} \cite{Li} in the usual way to the above apriori estimates. Also, these same  estimates guarantee that  star-shapedness and strict mean convexity are preserved along the flow.
We now complete the asymptotic analysis by showing that the first and second order  derivatives of $\tilde u$ actually vanish at infinity as indicated in (\ref{sharpest}). A key ingredient here is a sharper control on the mean curvature.

\begin{proposition}\label{vanish}
We have
\begin{equation}\label{vanishwh}
\frac{W}{H}= \frac{1}{n-1} + O(te^{-\frac{2t}{n-1}}).
\end{equation}
In particular,
\begin{equation}\label{vanish2}
H=n-1 +O(te^{-\frac{2t}{n-1}}).
\end{equation}
\end{proposition}

\begin{proof}
As explained in \cite{BHW}, it follows from the evolution equation for the uniformly bounded quantity $\tau=W/H=\lambda_\epsilon/pH$ that
\[
\frac{d}{dt}\log \overline\tau\leq \frac{2}{n-1}-2\overline \tau +O(e^{-\frac{2t}{n-1}}),
\]
so that
\begin{eqnarray*}
\frac{d\overline \tau}{dt} & \leq & \frac{2\overline \tau}{n-1}-2\overline \tau^2+O(e^{-\frac{2t}{n-1}})\\
&\leq & O(e^{-\frac{2t}{n-1}}),
\end{eqnarray*}
whenever $\tau\geq 1/(n-1)$. This gives
\[
\frac{W}{H}\leq \frac{1}{n-1}+ O(te^{-\frac{2t}{n-1}}),
\]
and the reverse inequality follows from (\ref{conth2}) and (\ref{supernew}).
\end{proof}

\begin{remark}\label{huwrem}
{\rm Using the  arguments in \cite{BHW} it can be shown that (\ref{vanish2}) leads to a rather precise control on the shape operator, namely,
\[
|a^i_j-\delta^i_j|\leq O(t^2e^{-\frac{2t}{n-1}}).
\]
In words, the evolving hypersurface
 becomes totally umbilical at an exponential rate.}
\end{remark}

So far we have essentially  followed \cite{BHW} and \cite{D}. Inspired by \cite{Ge1} we now introduce the uniformly bounded function
\begin{equation}\label{defw}
w=(e^u+1)e^{-\frac{t}{n-1}}=e^{\tilde u}+e^{-\frac{t}{n-1}}.
\end{equation}

\begin{proposition}\label{limit}
The functions $w$ and $\tilde u$ converge pointwisely as $t\to +\infty$.
\end{proposition}

\begin{proof}
By (\ref{defw}) it suffices to prove the assertion for $w$.
By (\ref{par}) we get
\begin{equation}\label{evolvew}
\frac{\partial w}{\partial t}  =  \frac{W}{H}{e^u}e^{-\frac{t}{n-1}}-\frac{w}{n-1},
\end{equation}
so that (\ref{vanishwh}) and (\ref{c0est2}) give
\[
\frac{\partial w}{\partial t}  =
   \frac{e^u}{n-1}e^{-\frac{t}{n-1}}-\frac{w}{n-1}+ O(te^{-\frac{2t}{n-1}}),
\]
and hence,
\begin{equation}\label{evolvv2}
\frac{\partial w}{\partial t} =
 -\frac{1}{n-1}e^{-\frac{t}{n-1}}+O(te^{-\frac{2t}{n-1}}).
\end{equation}
Thus, for all $t$ large enough there holds
\[
\frac{\partial w}{\partial t}\geq -\frac{2}{n-1}e^{-\frac{t}{n-1}},
\]
or
\[
\frac{\partial}{\partial t}\left(w-e^{-\frac{2t}{n-1}}\right)\geq 0.
\]
The assertion follows since $w=O(1)$.
\end{proof}

We now prove that the limiting functions above are actually  constant. This is the content of the next result, which also proves (\ref{sharpest}) and concludes the proof of Theorem \ref{longtime}.

\begin{proposition}\label{final}
There holds
\begin{equation}\label{estfinal}
|D w|_h+|D^2 w|_h=o(1).
\end{equation}
\end{proposition}

\begin{proof}
First note that (\ref{change2}), (\ref{meancur}) and our previous estimates imply that (\ref{evolvew}) is a {\em uniformly} parabolic equation for $w$. Also, (\ref{apc123}) and (\ref{evolvv2}) imply that both $\partial w/\partial t$ and $|Dw|_h$ decay exponentially. Since $w$ remains uniformly bounded in time, parabolic regularity \cite{Ge2} \cite{Li} implies that all higher order derivatives of $w$ are uniformly bounded as well. Standard interpolation inequalities then guarantee that $|D^2w|_h$ decays exponentially, which proves (\ref{estfinal}).
\end{proof}

\section{The proofs of Theorems \ref{aftype} and \ref{theomain}}
\label{proof}

As remarked in the Introduction, the Penrose inequality (\ref{penrose}) in Theorem \ref{theomain} is an immediate consequence of the Alexandrov-Fenchel-type inequality (\ref{aftype2}) in Theorem \ref{aftype}. We now explain how the above results on the asymptotics of the inverse mean curvature flow may be used to prove  (\ref{aftype2}). A first ingredient is a Heintze-Karcher-type inequality due to Brendle \cite{Br}.

\begin{proposition}\label{bren}
If $\Sigma\subset P_\epsilon$ is star-shaped and strictly mean convex then
\begin{equation}\label{bren2}
(n-1)\int_\Sigma\frac{\dot\lambda_\epsilon}{H}d\Sigma\geq \int_\Sigma pd\Sigma.
\end{equation}
Moreover, if the equality holds then $\Sigma$ is totally umbilical.
\end{proposition}

\begin{proof}
The static warped product $(P_\epsilon,g_\epsilon)$, with $g_\epsilon$ given by (\ref{change}), has $\rho_\epsilon=\dot\lambda_\epsilon$ as its static potential and satisfies the general structure conditions in \cite{Br}.
\end{proof}

We now determine the asymptotic behavior of certain integral quantities associated to solutions of the flow (\ref{imcfpar}). If $\Sigma\subset P_\epsilon$ is closed we define
\[
\Ic(\Sigma)=\int_\Sigma \rho_\epsilon Hd\Sigma=\int_\Sigma \dot\lambda_\epsilon Hd\Sigma,
\]
\[
\Jc(\Sigma)=\int_\Sigma pd\Sigma),
\]
\[
\Kc(\Sigma)=\vartheta_{n-1}\Ac(\Sigma)^{\frac{n}{n-1}},
\]
where $\Ac(\Sigma)={|\Sigma|}/{\vartheta_{n-1}}$ is the normalized area of $\Sigma$,
and
\[
\Lc(\Sigma)=\frac{\Ic(\Sigma)-(n-1)\Kc(\Sigma)}{\Ac(\Sigma)^{\frac{n-2}{n-1}}}.
\]

\begin{proposition}\label{asymbe}
If $\Sigma_t$ is a solution of (\ref{imcfpar}) with $\Sigma_0$ star-shaped and strictly mean convex then
\begin{equation}\label{asymbe1}
\lim_{t\to +\infty}\frac{\Jc}{\Ac^{\frac{n}{n-1}}}(\Sigma_t)=\vartheta_{n-1},
\end{equation}
and
\begin{equation}\label{asymbe2}
\liminf_{t\to +\infty}\Lc(\Sigma_t)\geq (n-1)\vartheta_{n-1}\epsilon.
\end{equation}
\end{proposition}

\begin{proof}
First note that (\ref{sharpest}) gives
\[
|D u|_h+|D^2u|_h=o(1),
\]
and hence
\[
|D v|_h+|D^2v|_h=o(e^{-\frac{t}{n-1}}).
\]
Thus,
\[
W^{-1}=1-\frac{1}{2}|D v|_h^2 +o(e^{-\frac{4t}{n-1}}),
\]
so that (\ref{supp}) gives
\[
p=\lambda_\epsilon-\frac{\lambda_\epsilon}{2}|D v|_h^2+o(e^{-\frac{3t}{n-1}}).
\]
On the other hand, (\ref{met1}) gives
\begin{equation}\label{expvol}
\sqrt{\det \eta}=\lambda_\epsilon^{n-1}\sqrt{\det h}\left(1+\frac{1}{2}|D v|_h^2+o(e^{-\frac{4t}{n-1}})\right).
\end{equation}
We thus obtain the expansions
\[
\Jc(\Sigma_t)=\vartheta_{n-1}\fint\lambda_\epsilon^n+o(e^{\frac{(n-4)t}{n-1}}),
\]
\begin{equation}\label{exp2}
\Ac(\Sigma_t)=\fint \lambda_\epsilon^{n-1}+o(e^{\frac{(n-3)t}{n-1}})
\end{equation}
and
\[
\Ac(\Sigma_t)^{\frac{n}{n-1}}=\left(\fint \lambda_\epsilon^{n-1}\right)^{\frac{n}{n-1}}+
o(e^{\frac{(n-2)t}{n-1}}),
\]
where
$
\fint
=\frac{1}{\vartheta_{n-1}}\int
$
and $\int$ means  integration  over $N$ with respect to $h$.
In particular,
\begin{equation}\label{inp}
\frac{\Jc}{\Ac^{\frac{n}{n-1}}}=\frac{\vartheta_{n-1}\fint\lambda_\epsilon^n}
  {\left(\fint\lambda_\epsilon^{n-1}\right)^{\frac{n}{n-1}}+o(e^{\frac{(n-2)t}{n-1}})}
  +o(e^{-\frac{4t}{n-1}}).
\end{equation}
Using standard mean value theorems it is easy to check that
\[
\left|\frac{\fint\lambda_\epsilon^n}{\left(
 \fint\lambda_\epsilon^{n-1}\right)^{\frac{n}{n-1}}}-1\right|\leq C\sup_N|D v|_h=o(1),
\]
so that
\begin{equation}\label{same}
\lim_{t\to +\infty}\frac{\fint\lambda_\epsilon^n}{\left(
 \fint\lambda_\epsilon^{n-1}\right)^{\frac{n}{n-1}}}=1,
\end{equation}
and this together with (\ref{inp}) proves (\ref{asymbe1}).

We now prove (\ref{asymbe2}). From (\ref{meancur}) we get
\[
\dot\lambda_\epsilon H=(n-1)\frac{\dot\lambda_\epsilon^2}{\lambda_\epsilon}+o(e^{-\frac{t}{n-1}}),
\]
and recalling (\ref{expvol}) we obtain
\[
\int_\Sigma \rho_\epsilon H d\Sigma=(n-1)\int\dot\lambda_\epsilon^2\lambda_\epsilon^{n-2} +o(e^{\frac{(n-2)t}{n-1}}).
\]
On the other hand, from (\ref{exp2}) we have
\[
\Ac(\Sigma_t)^{\frac{n-2}{n-1}}=\left(\fint \lambda_\epsilon^{n-1}\right)^{\frac{n-2}{n-1}}+o(e^{\frac{(n-4)t}{n-1}}).
\]
Hence, from (\ref{usef}) we get
\begin{eqnarray*}
\liminf_{t\to +\infty} \Lc(\Sigma_t) & = & (n-1)\vartheta_{n-1}\liminf_{t\to +\infty}
   \frac{\fint\dot\lambda_\epsilon^2\lambda_\epsilon^{n-2}-
     \left(\fint\lambda_\epsilon^{n-1}\right)^{\frac{n}{n-1}}+o( e^{\frac{(n-2)t}{n-1}})}{\left(\fint \lambda_\epsilon^{n-1}\right)^{\frac{n-2}{n-1}}+o(e^{\frac{(n-4)t}{n-1}})}\\
& \geq & (n-1)\vartheta_{n-1}\epsilon\liminf_{t\to +\infty}
    \frac{\fint\lambda_\epsilon^{n-2}}{\left(\fint \lambda_\epsilon^{n-1}\right)^{\frac{n-2}{n-1}}+o(e^{\frac{(n-4)t}{n-1}})}+\\
    & & \quad + (n-1)\vartheta_{n-1}\liminf_{t\to +\infty}\frac{\fint \lambda_\epsilon^n-
     \left(\fint \lambda_\epsilon^{n-1}\right)^{\frac{n}{n-1}}}
     {\left(\fint \lambda_\epsilon^{n-1}\right)^{\frac{n-2}{n-1}}+o(e^{\frac{(n-4)t}{n-1}})}+\\
     & & \quad\quad +(n-1)\vartheta_{n-1}\liminf_{t\to +\infty}
     \frac{o(e^{\frac{n-2}{n-1}})}{\left(\fint \lambda_\epsilon^{n-1}\right)^{\frac{n-2}{n-1}}+o(e^{\frac{(n-4)t}{n-1}})}.
\end{eqnarray*}
Since $\lambda_\epsilon(u)\sim e^{\frac{t}{n-1}}$, the last limit on the right-hand side vanishes. Also, the argument leading to (\ref{same}) shows that the first one is $1$. Finally, by Jensen's inequality,
\[
\left(\fint\lambda_\epsilon^{n-1}\right)^{\frac{n}{n-1}}\leq\fint (\lambda_\epsilon^{n-1})^{\frac{n}{n-1}}=\fint\lambda_\epsilon^n,
\]
which implies that the second limit is nonnegative. This proves (\ref{asymbe2}).
\end{proof}

We now identify monotone quantities
along solutions of the inverse mean curvature flow.

\begin{proposition}\label{mono}
Along solutions of (\ref{imcfpar}) we have
\begin{equation}\label{mono1}
\frac{d}{dt}\left(\frac{\Jc-\Kc}{\Ac^{\frac{n}{n-1}}}\right)\geq 0,
\end{equation}
and
\begin{equation}\label{mono2}
\frac{d\Lc}{dt}\leq 0.
\end{equation}
Moreover, if the  equality holds in any of these inequalities then $\Sigma$ is totally umbilical.
\end{proposition}

\begin{proof}
From (\ref{evolvarea}), (\ref{evolvsuppint}) and (\ref{bren2}) we get
\[
\frac{d}{dt}\left(\Jc-\Kc\right)\geq \frac{n}{n-1}\left(\Jc-\Kc\right),
\]
which immediately implies (\ref{mono1}).
In particular, if we let a star-shaped and strictly mean convex hypersurface $\Sigma\subset P_\epsilon$ flow under (\ref{imcfpar}) and use (\ref{asymbe1}) we deduce that
\begin{equation}\label{indep}
\Jc(\Sigma)\leq \Kc(\Sigma).
\end{equation}
Moreover, if the equality holds then $\Sigma$ is totally umbilical.

On the other hand, a straightforward computation using Proposition \ref{evolequations} gives
\[
\frac{d\Ic}{dt}=2\int_\Sigma\frac{\rho K}{H}d\Sigma +2\Jc,
\]
where $K$ is the {\em extrinsic} scalar curvature of $\Sigma$, i.e. the sum of products of its principal curvatures. It then follows from Newton-McLaurin's inequality that
\[
\frac{d\Ic}{dt}\leq \frac{n-2}{n-1}\Ic +2\Jc,
\]
so that
\begin{eqnarray*}
\frac{d}{dt}\left(\Ic-(n-1)\Kc\right) & \leq & \frac{n-2}{n-1}(\Ic-(n-1)\Kc) +2(\Jc-\Kc)\\
  &\leq & \frac{n-2}{n-1}(\Ic-(n-1)\Kc),
\end{eqnarray*}
where we used (\ref{indep}) in the last step.
This easily yields (\ref{mono2}) and completes the proof.
\end{proof}

The proof of Theorem \ref{aftype} now follows straightforwardly if we  let $\Sigma$ flow under (\ref{imcfpar}). Combining (\ref{mono2}) and (\ref{asymbe2})
we obtain  $\Lc(\Sigma)\geq (n-1)\vartheta_{n-1}\epsilon$, which is just a rewriting of the Alexandrov-Fenchel-type inequality (\ref{aftype2}). Moreover, if the equality holds then $\Sigma$ is totally umbilical. In particular, $\Sigma$ has constant mean curvature and results in \cite{Mo}  imply that $\Sigma$ is a slice. This proves Theorem \ref{aftype}. Finally, since any mean convex hypersurface can be arbitrarily approximated by strictly mean convex ones, the Penrose-type inequality (\ref{penrose}) holds. In case of  equality, it follows from (\ref{massform}) that $\Sigma$ satisfies
\begin{equation}\label{ellip}
R_g=-n(n-1)
\end{equation}
everywhere. As explained in \cite{dLG4}, the argument in \cite{HW}
can be adapted to show that $\Sigma$ is an elliptic solution of (\ref{ellip}).
Thus, Alexandrov reflection can be used in the vertical direction to make sure it is congruent to the graph realization of $(P_{\epsilon,m},g_{\epsilon,m})$ described in (\ref{embed}), which is an elliptic solution as well. This completes the proof of Theorem \ref{theomain}.

\end{document}